\documentclass{article}

\usepackage{multicol}
\usepackage{amssymb,latexsym,amsmath,amsthm,amsfonts} 
\usepackage{graphicx}
\usepackage{tikz}
\usepackage{enumitem}
\usepackage{verbatim}
\usepackage{color}
\usepackage{xcolor}
\usepackage{mathrsfs}
\usepackage[normalem]{ulem}
\usepackage{hyperref}
\usepackage{fancyvrb}
\usepackage{fancyhdr}
\usepackage{listings}
\usepackage{centernot}
\usepackage{mathtools}

\newtheorem{theorem}{Theorem}

\theoremstyle{definition}

\newtheorem{remark}[theorem]{Remark}
\newtheorem{claim}{Claim}

\usepackage{subfiles}
\usepackage{float}
\newcommand{\A}{\mathcal{A}}

\newcommand{\I}{\mathcal{I}}
\newcommand{\s}{\mathbf{s}}

\newcommand{\vi}{v^i}
\newcommand{\ir}{\mathcal{I}^{(r)}}

\pgfdeclarelayer{background}
\pgfdeclarelayer{foreground}
\pgfsetlayers{background,main,foreground}

\begin{document}

\title{EKR-Type Theorems for Pendant Graph Constructions}

\author{
Michael Carrion \footnotemark[1]\
\and
Melissa M. Fuentes \footnotemark[2]\ 
\and
Zaphenath Joseph \footnotemark[3]\
\and
Alexander Nappo \footnotemark[4]
}
\date{}
\maketitle

\footnotetext[1]{
Department of Mathematics \& Statistics, 
Villanova University, 
Villanova, PA, USA, 
\texttt{mcarri04@villanova.edu}
}

\footnotetext[2]{
Department of Mathematics \& Statistics, 
Villanova University, 
Villanova, PA, USA, 
\texttt{melissa.fuentes@villanova.edu}
}

\footnotetext[3]{
Department of Mathematics \& Statistics, 
Villanova University, 
Villanova, PA, USA, 
\texttt{zjoseph@villanova.edu}
}

\footnotetext[4]{
Department of Mathematics \& Statistics, 
Villanova University, 
Villanova, PA, USA, 
\texttt{anappo01@villanova.edu}
}

%\begin{multicols}{2}

%%%%%%%%%%%%%%%%%%%%%%%%%%%%%%%%%%%%%%%%%%%%%%%
%       ABSTRACT
%%%%%%%%%%%%%%%%%%%%%%%%%%%%%%%%%%%%%%%%%%%%%%%
\begin{abstract}
The classical Erd\H{o}s--Ko--Rado (EKR) theorem characterizes the maximum size of intersecting families of $r$-element subsets of an $n$-element set. We study EKR-type questions for independent $r$-sets in \emph{pendant} graph constructions, obtained by attaching to each base vertex a clique of prescribed size.

Our contributions are threefold. We give an alternate and purely combinatorial proof (via shifting and shadows) that the pendant complete graph $K_n^{*}$ is $r$-EKR for $n \ge 2r$, and strictly so for $n>2r$, recovering a result of De Silva, Dionne, Dunkelberg, and Harris. We extend this to \emph{generalized pendant complete graphs}, where every base vertex in the clique supports a clique of arbitrary size, proving that that generalized pendant complete graphs are $r$-EKR whenever $n \ge 2r$. For pendant paths $P_n^{*}$, we provide elementary constructions showing that $P_n^{*}$ is not $(n-k)$-EKR when $n \ge 3k+2$ for $k\ge 2$, not $(n-1)$-EKR for $n\ge 6$, and not $n$-EKR for $n\ge 4$. These results fit naturally into the Holroyd--Talbot perspective relating $r$-EKR thresholds to independence parameters and supply tools for further pendant constructions.
\end{abstract}

%%%%%%%%%%%%%%%%%%%%%%%%%%%%%%%%%%%%%%%%%%%%%%%
%       INTRO
%%%%%%%%%%%%%%%%%%%%%%%%%%%%%%%%%%%%%%%%%%%%%%%

\section{Introduction}

The classical Erd\H{o}s--Ko--Rado (EKR) theorem is a cornerstone of extremal combinatorics, describing the structure of intersecting families of uniform sets. In graph-theoretic extensions, one studies independent $r$-sets of a graph $G$ and asks when stars are extremal. Following Holroyd, Spencer, and Talbot~\cite{Holroyd2005}, we say that $G$ has the \emph{$r$-EKR property} if every largest intersecting subfamily of independent $r$-sets is an $r$-star (that is, all members contain a common vertex).\\

\noindent Throughout, graphs are finite and simple. For a graph $G$ and $r\in\mathbb{N}$, let $\ir(G)$ denote the family of independent $r$-subsets of $V(G)$. A subfamily $\A\subseteq\ir(G)$ is \emph{intersecting} if every two members have nonempty intersection. For $v\in V(G)$, we write $\ir_v(G):=\{A\in\ir(G):v\in A\}$ and refer to this as the \emph{$r$-star centered at $v$}. We also write $[n]=\{1,\dots,n\}$. \\

\noindent In this work we focus on \emph{pendant} constructions. Given a graph $G$ with vertex set $V(G)=\{v_1,\dots,v_n\}$, the \emph{pendant graph} $G^*$ is obtained by adjoining a new pendant vertex $p_i$ to each $v_i$:
\[
V(G^*) \;=\; V(G)\;\sqcup\;\{p_1,\dots,p_n\},
\qquad
E(G^*) \;=\; E(G)\;\sqcup\;\{\,v_ip_i:\;i\in [n]\,\}.
\]
We call $\{p_1,\dots,p_n\}$ the \emph{pendant vertices} of $G^*$. More generally, for a sequence $\s=(s_1,\dots,s_n)$ of positive integers, the \emph{generalized pendant graph} $G^{\s}$ is formed by attaching to $v_i$ a clique $K_{s_i}$ for each $i \in [n]$ with vertex set $V(K_{s_i})=\{\vi_1,\dots,\vi_{s_i}\}$, each adjacent to $v_i$ and to no other base vertex:
\[
V(G^{\s}) \;=\; V(G)\;\sqcup\; \bigl(\bigsqcup_{i=1}^n V(K_{s_i})\bigr),\qquad
E(G^{\s}) \;=\; E(G)\;\sqcup\;\bigl(\bigsqcup_{i=1}^n E(K_{s_i})\bigr)\;\sqcup\;\{\,v_i\vi_j:\; i\in[n],\,1\le j\le s_i\,\}.
\]
Let $\mathbf{1}=(1, \ldots, 1)$ be of length $n$. For $m\ge 1$ let $m\mathbf{1}=(m,\dots,m)$ and write $G^m:=G^{m\mathbf{1}}$; in particular $G^*=G^{\mathbf{1}}$. \\

\noindent Two benchmark families already illustrate the range of behaviors we study. For the pendant complete graph $K_n^*$, De~Silva, Dionne, Dunkelberg, and Harris proved that $K_n^*$ is $r$-EKR when $n\ge 2r$, and strictly so when $n>2r$~\cite[Theorem ~4]{DeSilvaDDH2023}. We provide an alternative proof via shifting and shadow arguments. For pendant paths $P_n^*$, the largest $r$-stars are centered at $p_2$ and $p_{n-1}$ for all $1\le r\le n$, and $P_n^*$ is not $n$-EKR for $n\ge 4$~\cite[Theorem 5 and Lemma 9]{DeSilvaDDH2023}. Building on these foundations, we extend the $K_n^*$ analysis to generalized pendant complete graphs $K_n^{\s}$ and also give explicit non-EKR constructions for $P_n^*$ at high values of $r$. 

\paragraph{Organization.}
Section~\ref{sec3} gives a short, purely combinatorial proof that $K_n^*$ is $r$-EKR for $n\ge 2r$ (with strictness for $n>2r$). Section~\ref{sec4} extends this to generalized pendant complete graphs $K_n^{\s}$, proving $r$-EKR under the same threshold. Section~\ref{sec5} turns to pendant paths $P_n^*$ and presents elementary constructions showing that $P_n^{*}$ is not $(n-k)$-EKR when $n \ge 3k+2$ for $k\ge 2$, not $(n-1)$-EKR for $n\ge 6$, and not $n$-EKR for $n\ge 4$. Section~\ref{sec6} outlines directions for further work, including pendant cycles and powers of paths or cycles.

\section{Complete Pendant Graphs}
\label{sec3}

The EKR property for the pendant complete graph $K_n^*$ was established by De~Silva, Dionne, Dunkelberg, and Harris~\cite{DeSilvaDDH2023}. Here we present an alternative proof using classical shifting and shadow techniques, inspired by Katona's shadow bound~\cite{Katona1964} and a short proof of classical EKR due to Frankl and F\"uredi~\cite{FranklFuredi}. The approach is purely combinatorial. \\

\noindent Given a family $\A$ of $r$-element sets and an integer $s$, the \emph{$s$-shadow} of $\A$ is
\[
\partial_s \A \;=\; \{\,S:\ |S|=s \text{ and } S\subseteq A \text{ for some } A\in\A\,\}.
\]
Shadows can be used to bound intersecting families via the following result of Katona.

\begin{theorem}[Katona, ~\cite{Katona1964}]\label{thm:Katona}
Let $a,b$ be integers with $0\le b\le a$, and let $\A$ be a family of $a$-element sets such that
$|A\cap A'|\ge b$ for all $A,A'\in\A$. Then
\[
|\A|\ \le\ |\partial_{\,a-b}\A|\,.
\]
\end{theorem}

\begin{theorem}\label{thm:KnstarEKR}
For all $n \ge 2r$, the pendant complete graph $K_n^*$ is $r$-EKR. Moreover, when $n>2r$, equality holds only for an $r$-star centered at a pendant vertex.
\end{theorem}

\begin{proof}
Write $V(K_n^*)=V(K_n)\sqcup P$ with $V(K_n)=\{v_1,\dots,v_n\}$ and $P=\{p_1,\dots,p_n\}$, where $p_i$ is pendant to $v_i$. Let $\A\subseteq \I^{(r)}(K_n^*)$ be intersecting. We prove
\[
|\A|\ \le\ r\binom{n-1}{r-1},
\]
with the indicated equality condition for $n>2r$.

For $i\in[n]$ define the shift
\[
S_i(A)\;=\;
\begin{cases}
(A\setminus\{v_i\})\cup\{p_i\}, &\text{if } v_i\in A \text{ and } (A\setminus\{v_i\})\cup\{p_i\}\notin\A,\\
A,&\text{otherwise}.
\end{cases}
\]
Iterating $S_1,\dots,S_n$ we may assume $\A$ is \emph{shift-stable}, i.e., $S_i(\A)=\A$ for all $i$.

Since $K_n$ is a clique, any independent $r$-set contains at most one base vertex $v_i$. Partition $\A$ into sets
\[
\A_0 \;=\; \{A\in\A:\ A\cap V(K_n)=\emptyset\} \text{ and }
\A_1 \;=\; \{A\in\A:\ |A\cap V(K_n)|=1\}.
\]
The classical EKR Theorem (\cite{EKR}) yields
\begin{equation} \label{classicEKR}
|\A_0|\ \le\ \binom{n-1}{r-1}
\end{equation}
for $n\ge 2r$.

Each $A\in\A_1$ has the form $A=A'\cup\{v_i\}$ with $A'\subseteq P$ and $|A'|=r-1$. Define
\[
\A_1' \;=\; \{\,A\setminus\{v\}:\ A\in\A_1,\ v\in A\cap V(K_n)\,\}.
\]

\begin{claim}\label{cl:intersecting}
$\A_1'$ is intersecting.
\end{claim}
\begin{proof}
Suppose $A,B\in \A_1'$ with $A\cap B=\emptyset$. Then $A\cup\{v_i\},\,B\cup\{v_j\}\in \A_1$ for some $i,j$. If $i\ne j$, then $(A\cup\{v_i\})\cap (B\cup\{v_j\})=\emptyset$, contradicting that $\A_1$ is intersecting. Hence $i=j$. By shift-stability, $(B\cup\{v_i\}\setminus\{v_i\})\cup\{p_i\}=B\cup\{p_i\}\in \A$. However,
\[
(A\cup\{v_i\})\cap (B\cup\{p_i\})=\emptyset
\]
since $A\cap B=\emptyset$, $v_i\notin B\cup\{p_i\}$, and $p_i\notin A\cup\{v_i\}$, again a contradiction. Thus $\A_1'$ is intersecting.
\end{proof}

\medskip

Fix $p_1\in P$ and put $P'=P\setminus\{p_1\}$. Form the partition
\[
\A_1'\;=\;\mathcal{P}_1'\ \cup\ \overline{\mathcal{P}_1'},
\]
where
\[
\mathcal{P}_1'=\{X\in \A_1':\,p_1\in X\} \quad \text{ and }
\quad
\overline{\mathcal{P}_1'}=\{X\in \A_1':\,p_1\notin X\}.
\]
Define
\[
\mathcal{P}_1''=\{X\setminus\{p_1\}: X\in \mathcal{P}_1'\}\,
\text{ and }
\overline{\mathcal{P}_1''}=\{\,P'\setminus X:\ X\in \overline{\mathcal{P}_1'}\,\}.
\]
Note that sets in $\mathcal{P}_1''$ are of size $r-2$ and sets in $\overline{\mathcal{P}_1''}$ are of size $n-r$. Clearly, $|\mathcal{P}_1'|=|\mathcal{P}_1''|$ and $|\overline{\mathcal{P}_1'}|=|\overline{\mathcal{P}_1''}|$.

\bigskip

\begin{claim}\label{cl:katona-parameters}
Every two sets of $\overline{\mathcal{P}_1''}$ have intersection size at least $b:=n-2r+2$.
\end{claim}
\begin{proof}
Take $A,B\in \overline{\mathcal{P}_1''}$ and let $A',B'\in \overline{\mathcal{P}_1'}$ correspond so that $A=P'\setminus A'$ and $B=P'\setminus B'$. Then
\[
|A\cap B| = |P'\setminus (A'\cup B')|
= |P'|-|A'|-|B'|+|A'\cap B'|
= (n-1)-2(r-1)+|A'\cap B'|.
\]
By Claim~\ref{cl:intersecting}, $\overline{\mathcal{P}_1'}$ is intersecting, so $|A'\cap B'|\ge 1$. Hence $|A\cap B|\ge n-2r+2$, as claimed.
\end{proof}

Applying Theorem~\ref{thm:Katona} to the $a$-uniform family $\overline{\mathcal{P}_1''}$ with $a=n-r$ and $b=n-2r+2$ (so $a-b=r-2$) gives
\[
|\overline{\mathcal{P}_1''}|\ \le\ |\partial_{\,r-2}(\overline{\mathcal{P}_1''})|.
\]

\bigskip

\begin{claim}\label{cl:disjoint}
$\mathcal{P}_1'' \,\cap\, \partial_{\,r-2}(\overline{\mathcal{P}_1''})=\emptyset$.
\end{claim}
\begin{proof}
Suppose $X\in \mathcal{P}_1''$ and $U\in \partial_{\,r-2}(\overline{\mathcal{P}_1''})$. Then $X\subseteq P'\setminus B'$ for some $B'\in \overline{\mathcal{P}_1'}$, so $X\cap B'=\emptyset$. However, $X\cup\{p_1\}\in \mathcal{P}_1'\subseteq \A_1'$ and $B'\in \overline{\mathcal{P}_1'}\subseteq \A_1'$, with $(X\cup\{p_1\})\cap B'=\emptyset$, contradicting Claim~\ref{cl:intersecting}.
\end{proof}

\bigskip

Both families $\mathcal{P}_1''$ and $\partial_{\,r-2}(\overline{\mathcal{P}_1''})$ consist of $(r-2)$-subsets of $P'$, and by Claim~\ref{cl:disjoint}, they are disjoint. Therefore,
\[
|\mathcal{P}_1''| \;+\; |\partial_{\,r-2}(\overline{\mathcal{P}_1''})|
\ \le\ \binom{n-1}{r-2}.
\]
Thus,
\[
|\A_1'| \;=\; |\mathcal{P}_1'|+|\overline{\mathcal{P}_1'}|
\;=\; |\mathcal{P}_1''|+|\overline{\mathcal{P}_1''}|
\ \le\ |\mathcal{P}_1''|+|\partial_{\,r-2}(\overline{\mathcal{P}_1''})|
\ \le\ \binom{n-1}{r-2}.
\]

Each $(r-1)$-set in $\A_1'$ forbids its corresponding $r-1$ base vertices in $V(K_n)$, so it extends to an independent $r$-set in at most $n-r+1$ ways by adding one base vertex. Hence,
\[
|\A_1|\ \le\ (n-r+1)\,|\A_1'|
\ \le\ (n-r+1)\binom{n-1}{r-2}.
\]

Finally,
\[
|\A|
\;=\;
|\A_0|+|\A_1|
\ \le\ 
\binom{n-1}{r-1} + (n-r+1)\binom{n-1}{r-2}
\;=\;
r\binom{n-1}{r-1}.
\]

If $n>2r$, the EKR inequalities used for $|\A_0|$ in (\ref{classicEKR}) and for the shadow-bound component are strict unless the respective families are stars; tracing back through the equalities forces $\A$ to be exactly an $r$-star centered at a pendant vertex. When $n=2r$, non-star extremals can occur (as in classical EKR), so we do not claim strictness in that boundary case. This matches the behavior described in~\cite{DeSilvaDDH2023} in Theorem 4.
\end{proof}

%%%%%%%%%%%%%%%%%%%%%%%%%%%%%%%%%%%%%%%%%%%%%%%
%       GENERALIZED PENDANT COMPLETE GRAPHS
%%%%%%%%%%%%%%%%%%%%%%%%%%%%%%%%%%%%%%%%%%%%%%%

\section{Generalized Pendant Complete Graphs}\label{sec4}

In this section, we extend the $K_n^*$ result to pendant constructions over a complete base with larger pendant cliques. 
Given $K_n$ and a sequence $\s=(s_1,\dots,s_n)$ of positive integers, the \emph{generalized pendant complete graph} $K_n^{\s}$ is formed by attaching to each $v_i\in V(K_n)$ a clique $K_{s_i}$; the vertices of $K_{s_i}$ are adjacent to $v_i$ and to no other base vertex. The special cases are $K_n^{*}=K_n^{\mathbf{1}}$ and $K_n^m:=K_n^{m\mathbf{1}}$.\\

\noindent We will use the following result of Bollobás and Leader for the disjoint union of $n$ complete graphs $K_m$, which we denote by $n K_m$.

\begin{theorem}[Bollobás--Leader~\cite{BollLead}]\label{thm:BollobasLeader}
Let $m\ge 2$ and $1\le r\le n$. If $\A\subseteq \mathcal{I}^{(r)}(nK_m)$ is intersecting, then
\[
|\A|\ \le\ m^{\,r-1}\binom{n-1}{r-1}\,.
\]
\end{theorem}

\medskip

\begin{theorem}\label{thm:PendantKn}
Let $m\ge 2$. Then the generalized pendant complete graph $K_n^m$ is $r$-EKR for all $n\ge 2r$.
In particular, the case $m=1$ is $K_n^*$ and follows from Theorem~\ref{thm:KnstarEKR}.
\end{theorem}

\begin{proof}
Label the vertices of the pendant clique at $v_i$ by $V(K_m)=\{v_1^i,\dots,v_m^i\}$. Let $\A\subseteq \mathcal{I}^{(r)}(K_n^m)$ be intersecting.

For each base vertex $v_i$, define a shift
\[
S_i(A)=
\begin{cases}
(A\setminus\{v_i\})\cup\{v_1^i\}, &\text{if } v_i\in A \text{ and } (A\setminus\{v_i\})\cup\{v_1^i\}\notin \A,\\
A,&\text{otherwise}.
\end{cases}
\]
After iterating $S_1,\dots,S_n$, assume $\A$ is shift-stable. Partition $\A$ into two sets,
\[
\A_0=\{A\in\A:\ A\cap V(K_n)=\emptyset\} \quad \text{ and }
\quad
\A_1=\{A\in\A:\ |A\cap V(K_n)|=1\}.
\]
On the pendant layer $nK_m$, Theorem~\ref{thm:BollobasLeader} gives
\[
|\A_0|\ \le\ m^{\,r-1}\binom{n-1}{r-1}.
\]

For each $A\in\A_1$ remove its unique base vertex to obtain an $(r-1)$-set in $nK_m$:
\[
\A_1'=\{A\setminus\{v\}:\ A\in\A_1,\ v\in A\cap V(K_n)\}\ \subseteq \mathcal{I}^{(r-1)}(nK_m).
\]

\bigskip

\begin{claim}\label{cl:A1prime-int}
$\A_1'$ is intersecting.
\end{claim}
\begin{proof}
If $A,B\in \A_1'$ were disjoint, there exist $i,j$ with $A\cup\{v_i\},B\cup\{v_j\}\in\A_1$. If $i\ne j$ these are disjoint, contradicting that $\A$ is intersecting. If $i=j$, shift-stability implies $A\cup\{v_1^i\}\in\A$. Then 
\[
(A\cup\{v_1^i\})\cap (B\cup\{v_i\})=\emptyset,
\]
which is a contradiction.
\end{proof}

\bigskip

Applying Theorem~\ref{thm:BollobasLeader} to $\A_1'\subseteq \mathcal{I}^{(r-1)}(nK_m)$ yields
\[
|\A_1'|\ \le\ m^{\,r-2}\binom{n-1}{r-2}.
\]
Each $A\in\A_1'$ forbids exactly the $r-1$ base vertices paired with its chosen pendants, so $A$ extends to a member of $\A_1$ in at most $n-r+1$ ways. Hence
\[
|\A_1|\ \le\ (n-r+1)\,|\A_1'|\ \le\ (n-r+1)\,m^{\,r-2}\binom{n-1}{r-2}.
\]

Summing gives
\[
|\A|\ =\ |\A_0|+|\A_1|\ \le\ m^{\,r-1}\binom{n-1}{r-1}+(n-r+1)\,m^{\,r-2}\binom{n-1}{r-2}
= m^{\,r-2}(m+r-1)\binom{n-1}{r-1}.
\]

Consider the $r$-star at a canonical pendant, say $v_1^1$. By partitioning according to whether or not a set contains a base vertex, we have
\[
|\mathcal{I}^{(r)}_{v_1^1}(K_n^m)|
=
m^{\,r-1}\binom{n-1}{r-1}
\;+\;
(n-1)\,m^{\,r-2}\binom{n-2}{r-2}
\;=\; m^{\,r-2}(m+r-1)\binom{n-1}{r-1},
\]
so the bound is tight and $K_n^m$ is $r$-EKR for $n\ge 2r$.
\end{proof}

\medskip

\begin{theorem}\label{thm:general-s}
Let $\s=(s_1,\dots,s_n)$ be a sequence of positive integers and $n\ge 2r$. Then $K_n^{\s}$ is $r$-EKR.
\end{theorem}

\begin{proof}
By relabeling the base vertices, assume $s_1\le \cdots \le s_n$, and within the largest pendant clique $K_{s_n}$ fix distinct vertices $v_1^n,v_2^n$ (if all $s_i$ are equal, then $K_n^{\s}=K_n^m$ and the result follows from Theorem~\ref{thm:PendantKn} together with the $m=1$ case from Theorem~\ref{thm:KnstarEKR}). Let $\A\subseteq \mathcal{I}^{(r)}(K_n^{\s})$ be intersecting.

Define a \emph{local} shift $T$ on sets in $\A$ by
\[
T(A)=
\begin{cases}
(A\setminus\{v_1^n\})\cup\{v_2^n\}, &\text{if } v_1^n\in A \text{ and } (A\setminus\{v_1^n\})\cup\{v_2^n\}\notin \A,\\
A,&\text{otherwise}.
\end{cases}
\]
Replacing $\A$ by $T(\A)$ if necessary, we may assume $T(\A)=\A$ (shift-stable within $K_{s_n}$).

Partition $\A$ by the presence of $v_1^n$:
\[
\A_0=\{A\in\A:\ v_1^n\notin A\},\qquad
\A_1=\{A\in\A:\ v_1^n\in A\},\qquad
\A_1'=\{A\setminus\{v_1^n\}:\ A\in\A_1\}.
\]
Let $\s'=(s_1,\dots,s_n-1)$ and $\s''=(s_1,\dots,s_{n-1})$.

\bigskip

\begin{claim}\label{claim:intersectingA1prime}
$\A_1'\subseteq \mathcal{I}^{(r-1)}\!\bigl(K_{n-1}^{\s''}\bigr)$ is intersecting.
\end{claim}
\begin{proof}
Each $A'\in \A_1'$ is an $(r-1)$-set in $K_{n-1}^{\s''}$ since we delete the single vertex $v_1^n$. 
If $A_1',A_2'\in \A_1'$ were disjoint, let $A_i=A_i'\cup\{v_1^n\}\in\A_1$ ($i=1,2$). By $T$-stability, $A_1' \cup\{v_2^n\}\in\A$. But then 
\[
\bigl(A_1'\cup\{v_2^n\}\bigr)\cap \bigl(A_2'\cup\{v_1^n\}\bigr)=\emptyset,
\]
contradicting that $\A$ is intersecting.
\end{proof}

\bigskip

\begin{claim}\label{claim:counting-star-s}
For any pendant $v_1^1$,
\[
\bigl|\mathcal{I}^{(r)}_{v_1^1}(K_n^{\s})\bigr|
\;=\;
\bigl|\mathcal{I}^{(r)}_{v_1^1}(K_n^{\s'} )\bigr|
\;+\;
\bigl|\mathcal{I}^{(r-1)}_{v_1^1}(K_{n-1}^{\s''})\bigr|.
\]
\end{claim}
\begin{proof}
Partition $\mathcal{I}^{(r)}_{v_1^1}(K_n^{\s})$ by whether $v_1^n$ is present. Removing $v_1^n$ gives a bijection from the subfamily containing $v_1^n$ to $\mathcal{I}^{(r-1)}_{v_1^1}(K_{n-1}^{\s''})$, while the subfamily avoiding $v_1^n$ is precisely $\mathcal{I}^{(r)}_{v_1^1}(K_n^{\s'})$.
\end{proof}

\bigskip

We now argue by double induction on $(r,\sum_i s_i)$ (lexicographic order). The case $r=1$ is trivial. If $s_1=\cdots=s_n$, the result follows from Theorem~\ref{thm:PendantKn} (and Theorem~\ref{thm:KnstarEKR} when $s_i=1$ for every $i$). Otherwise $s_n\ge 2$ and we apply the local shift $T$ above.

Since $\A_0$ is an intersecting family of $r$-sets in $K_n^{\s'}$ and $\A_1'$ is an intersecting family of $(r-1)$-sets in $K_{n-1}^{\s''}$ by Claim~\ref{claim:intersectingA1prime}, the inductive hypothesis gives
\[
|\A_0|\ \le\ \bigl|\mathcal{I}^{(r)}_{v_1^1}(K_n^{\s'})\bigr|\,,\qquad
|\A_1'|\ \le\ \bigl|\mathcal{I}^{(r-1)}_{v_1^1}(K_{n-1}^{\s''})\bigr|.
\]
Since $|\A|=|\A_0|+|\A_1'|$, Claim~\ref{claim:counting-star-s} yields
\[
|\A|\ \le\ \bigl|\mathcal{I}^{(r)}_{v_1^1}(K_n^{\s'})\bigr|+\bigl|\mathcal{I}^{(r-1)}_{v_1^1}(K_{n-1}^{\s''})\bigr|
\ =\ \bigl|\mathcal{I}^{(r)}_{v_1^1}(K_n^{\s})\bigr|,
\]
and hence, $K_n^{\s}$ is $r$-EKR.
\end{proof}

%%%%%%%%%%%%%%%%%%%%%%%%%%%%%%%%%%%%%%%%%%%%%%%
%       COUNTEREXAMPLE FOR PENDANT PATH GRAPHS
%%%%%%%%%%%%%%%%%%%%%%%%%%%%%%%%%%%%%%%%%%%%%%%
\section{Counterexample for Pendant Path Graphs}\label{sec5}

\noindent In contrast to the EKR behavior observed in pendant complete graphs and their generalizations, the EKR property does not hold universally for all pendant constructions. In this section, we demonstrate that the pendant path graph $P_n^*$ fails to be $(n-k)$-EKR when $n \ge 3k + 2$ and $k\geq 2$ (we treat the $k=0$ and $k=1$ cases separately using subtly different constructions). Specifically, we construct an intersecting family of independent sets of size $n-k$ whose size exceeds that of the star family centered at $p_2$ in $P_n^*$, which De Silva et al. (\cite{DeSilvaDDH2023}) proved is the largest star center.\\

\noindent We acknowledge De Silva et al. had previously proven the pendant path graph fails to be $n$-EKR. However, we provide an alternate construction in Theorem \ref{notEKR3}. Our approach involves partitioning based on the value of $k$, and in each case, we exhibit explicit families that augment a standard leaf-centered star into a strictly larger intersecting family.

\bigskip

\begin{theorem}\label{notEKR1}
Let $n,k\in\mathbb{N}$ with $k\ge 2$. If $n\ge 3k+2$, then the pendant path graph $P_n^*$ is \emph{not} $(n-k)$‑EKR.
\end{theorem}

\begin{proof}
Let $\mathcal{L}_2:=\mathcal{I}^{(n-k)}_{p_2}(P_n^*)$, which is a largest $(n-k)$‑star by \cite[Theorem ~5]{DeSilvaDDH2023}.  Define the single $(n-k)$‑set
\[
T_k\ :=\ \{p_{k+1},p_{k+2},\dots,p_n\}\ \in\ \mathcal{I}^{(n-k)}(P_n^*).
\]
We claim that every $S\in\mathcal{L}_2$ intersects $T_k$ when $n\ge 3k+2$. Suppose not. Then $S\cap T_k=\emptyset$, so $S$ uses no pendants among indices $k+1,\dots,n$. Consequently, all vertices of $S$ with indices in $\{k+1,\dots,n\}$ must lie in the base set $\{x_{k+1},\dots,x_n\}$.

The subgraph induced by $\{x_{k+1},\dots,x_n\}$ is a path on $n-k$ vertices, whose independence number is $\alpha=\lceil (n-k)/2\rceil$. From indices $1,\dots,k$, the set $S$ contains $p_2$ and may include at most $k-1$ other pendants (one from each remaining base–pendant pair). Hence,
\[
|S|\ \le\ 1+(k-1)+\Big\lceil\frac{n-k}{2}\Big\rceil\ =\ k+\Big\lceil\frac{n-k}{2}\Big\rceil.
\]
Under $n\ge 3k+2$ we have $\lceil (n-k)/2\rceil\le (n-k)/2$ and therefore
\[
k+\Big\lceil\frac{n-k}{2}\Big\rceil\ \le\ k+\frac{n-k}{2}\ <\ n-k,
\]
contradicting that $|S|=n-k$. Thus $S\cap T_k\neq\emptyset$ for every $S\in\mathcal{L}_2$.

It follows that $\mathcal{F}:=\mathcal{L}_2\cup\{T_k\}$ is intersecting and strictly larger than the star $\mathcal{L}_2$ (since $T_k\notin\mathcal{L}_2$). Hence, $P_n^*$ is not $(n-k)$‑EKR for $n\ge 3k+2$, $k\ge 2$.
\end{proof}

\medskip

We next treat the case $k=1$.

\begin{theorem}\label{notEKR2}
For $n\ge 6$, the pendant path graph $P_n^*$ is \emph{not} $(n-1)$‑EKR.
\end{theorem}

\begin{proof}
Let $\mathcal{L}_2:=\mathcal{I}^{(n-1)}_{p_2}(P_n^*)$ be a largest $(n-1)$‑star by \cite[Thm.~5]{DeSilvaDDH2023}. Define the single $(n-1)$‑set
\[
T'\ :=\ \{p_1,p_3,p_4,\dots,p_n\}\ \in\ \mathcal{I}^{(n-1)}(P_n^*).
\]
We claim that $S\cap T'\neq\emptyset$ for every $S\in\mathcal{L}_2$ when $n\ge 6$. If $S\cap T'=\emptyset$, then $S$ contains $p_2$ and \emph{no other} pendant, so all remaining $n-2$ vertices of $S$ would have to lie in the base set $\{x_1,\dots,x_n\}$. Since a path on $n$ vertices has independence number $\lceil n/2\rceil$, we have
\[
|S|\ \le\ 1+\Big\lceil\frac{n}{2}\Big\rceil\ <\ n-1,
\]
for all $n \ge 6$, a contradiction. Therefore $\mathcal{L}_2\cup\{T'\}$ is intersecting and strictly larger than the star, and $P_n^*$ is not $(n-1)$‑EKR for $n\ge 6$.
\end{proof}

\medskip

Finally we address the case $k=0$ (i.e., $r=n$). 

\begin{theorem}\label{notEKR3}
For all $n\ge 4$, the pendant path graph $P_n^*$ is \emph{not} $n$‑EKR.
\end{theorem}

\begin{proof}
Define
\[
A\ :=\ \{x_i:\ i\ \text{odd}\}\ \cup\ \{p_i:\ i\ \text{even}\}\ \in\ \mathcal{I}^{(n)}(P_n^*).
\]
Then $A$ is independent: among the base vertices we select only $x_i$ with odd $i$, which are pairwise nonadjacent in the path $x_1x_2\cdots x_n$, and each $p_i$ is adjacent only to $x_i$. Let $A^{c}$ be the set obtained by swapping the choice in each pair $\{x_i,p_i\}$; explicitly,
\[
A^{c}\ =\ \{x_i:\ i\ \text{even}\}\ \cup\ \{p_i:\ i\ \text{odd}\}\ \in\ \mathcal{I}^{(n)}(P_n^*).
\]
Thus $A$ and $A^{c}$ are two \emph{disjoint} independent $n$-sets.

We claim that if $S,T\in \mathcal{I}^{(n)}(P_n^*)$ are disjoint, then $T$ is the \emph{pairwise complement} of $S$ (i.e., for each $i$, $S$ contains exactly one of $\{x_i,p_i\}$ and $T$ contains the other). Indeed, disjointness and $|S|=|T|=n$ force $S\cup T$ to contain exactly one vertex from each pair $\{x_i,p_i\}$ twice, hence one in $S$ and the other in $T$. In particular, both $S$ and $S^{c}$ must be independent. On a path, this happens if and only if the set of indices where $S$ chooses base vertices is an \emph{alternating} (parity) subset (so that no two consecutive base vertices are chosen in either $S$ or $S^{c}$). Consequently, up to swapping parity, the only disjoint pair within $\mathcal{I}^{(n)}(P_n^*)$ is $\{A, A^{c}\}$.

Consider now the family
\[
\mathcal{F}\ :=\ \mathcal{I}^{(n)}(P_n^*)\ \setminus\ \{A^{c}\}.
\]
By the preceding paragraph, $\mathcal{F}$ is intersecting: any two disjoint members of $\mathcal{I}^{(n)}(P_n^*)$ would have to be $A$ and $A^{c}$, but $A^{c}\notin \mathcal{F}$.

Finally, compare $|\mathcal{F}|$ with the size of an $n$‑star. With $\mathcal{L}_2:=\mathcal{I}^{(n)}_{p_2}(P_n^*)$, note that $\mathcal{L}_2$ contains only $n$‑sets that include $p_2$, whereas there are \emph{at least two} independent $n$‑sets that avoid $p_2$, namely $A^{c}$ and
\[
C\ :=\ \{x_2,x_4\}\ \cup\ \{p_i:\ i\notin\{2,4\}\}\ \in\ \mathcal{I}^{(n)}(P_n^*)\qquad(n\ge 4),
\]
since $x_2$ and $x_4$ are nonadjacent and each $p_i$ is adjacent only to $x_i$. Hence
\[
|\mathcal{L}_2|\ \le\ \bigl|\mathcal{I}^{(n)}(P_n^*)\bigr|\ -\ 2
\quad\text{while}\quad
|\mathcal{F}|\ =\ \bigl|\mathcal{I}^{(n)}(P_n^*)\bigr|\ -\ 1,
\]
so $|\mathcal{F}|>|\mathcal{L}_2|$. Therefore the star $\mathcal{L}_2$ is not maximum among intersecting families of independent $n$‑sets, and $P_n^*$ is not $n$‑EKR for all $n\ge 4$.
\end{proof}

\begin{remark}[Largest stars on $P_n^*$]
For completeness we record that, for each $1\le r\le n$, the largest $r$‑stars of $P_n^*$ are centered at $p_2$ and $p_{n-1}$; see \cite[Theorem 5]{DeSilvaDDH2023}. This validates the choice of the star $\mathcal{L}_2$ as the correct benchmark in Theorems~\ref{notEKR1}, \ref{notEKR2}, and \ref{notEKR3}.
\end{remark}

%%%%%%%%%%%%%%%%%%%%%%%%%%%%%%%%%%%%%%%%%%%%%%%
%       Future Directions 
%%%%%%%%%%%%%%%%%%%%%%%%%%%%%%%%%%%%%%%%%%%%%%%

\section{Future Directions} \label{sec6}

This paper established the Erdős--Ko--Rado (EKR) property for pendant constructions over complete graphs and their generalizations, and demonstrated its failure in certain cases for pendant path graphs. Several natural directions for future research arise from these results.\\

\noindent One compelling question is whether the pendant path graph $P_n^*$ satisfies the EKR property exactly for all $r \le \mu(P_n^*)/2 = n/2$ and fails as soon as $r > n/2$. In fact, we show that whenever $k \ge 2$ and $n \ge 3k+2$, one can build an intersecting family larger than the star at $p_2$, proving that $P_n^*$ is not $(n-k)$-EKR (including the cases $r=n$ for $n\ge4$, $r=n-1$ for $n\ge6$, and more generally $r=n-k\ge\lceil(2n+2)/3\rceil$). These explicit constructions refine and extend De Silva’s alternate counterexample and confirm that, while EKR holds for all $r \le n/2$, it breaks down immediately above that half-independence threshold—fully aligning with the Holroyd–Talbot conjecture.\\

\noindent Beyond paths, it is natural to ask whether other pendant constructions satisfy similar extremal properties. One such example is the \emph{pendant cycle graph} $C_n^*$, formed by attaching a pendant vertex to each vertex of the cycle $C_n$. Determining for which values of $r$ the graph $C_n^*$ is $r$-EKR remains an open question and may require new ideas beyond those used for pendant paths and cliques.\\

\noindent Another direction involves studying powers of paths and cycles. The $k$-th power of a graph $G$, denoted $G^k$, connects each vertex to all others within distance $k$ in the original graph. Holroyd, Spencer, and Talbot~\cite{Holroyd2005} proved that for certain values of $r$, the $k$-th powers of paths and cycles are $r$-EKR, particularly when $r \leq \alpha(G)$, where $\alpha(G)$ is the independence number. However, a complete understanding of which values of $r$ yield extremal star families in these powers is still lacking.\\

\noindent While these families share structural similarities, their EKR behavior appears to diverge in subtle but significant ways. The challenge ahead lies in identifying the precise graph features—such as symmetry, neighborhood overlap, and independence number—that govern whether stars are extremal. Developing techniques that can handle these nuances across broader classes of graphs remains a central problem in the study of EKR-type problems for families of independent sets in graphs.

%%%%%%%%%%%%%%%%%%%%%%%%%%%%%%%%%%%%%%%%%%%%%%%
%       ACKNOWLEDGMENTS
%%%%%%%%%%%%%%%%%%%%%%%%%%%%%%%%%%%%%%%%%%%%%%%
\section*{Acknowledgments}
The authors are deeply grateful to Dr.Vikram Kamat for his generous guidance and many constructive comments throughout this project. His suggestions, especially on the use of compression-style arguments and the organization of our pendant constructions, substantially improved the clarity and scope of our results. We also thank him for several careful readings of earlier drafts and for pointing out refinements that strengthened a number of proofs.

%%%%%%%%%%%%%%%%%%%%%%%%%%%%%%%%%%%%%%%%%%%%%%%
%           BIBLIOGRAPHY (REVISED)
%%%%%%%%%%%%%%%%%%%%%%%%%%%%%%%%%%%%%%%%%%%%%%%

\end{document}